\documentclass[twoside,11pt]{amsart}
\usepackage[utf8]{inputenc}
\usepackage[english]{babel}
\usepackage[fixlanguage]{babelbib}
\usepackage{color,graphicx}
\usepackage{amssymb,amsfonts,amsmath,amsthm,amscd, epsfig}
\usepackage[all]{xy}
\usepackage{tensor}
\usepackage{microtype}
\usepackage{tikz-cd}
\usepackage{float}

\title[Generalizations of Lagrange and Sylow Theorems]{Generalizations of Lagrange and Sylow Theorems for Groupoids}
\author[Beier, Garcia, Lautenschlaeger, Pedrotti and Tamusiunas]{Gustav Beier, Christian Garcia, Wesley G. Lautenschlaeger, \\ Juliana Pedrotti and Thaísa Tamusiunas}
\address{Instituto de Matem\'{a}tica, Universidade Federal do Rio Grande do Sul,  Av. Bento Gon\c{c}alves, 9500, 91509-900. Porto Alegre-RS, Brazil}
\email{gustavbeier@gmail.com}
\email{christian.garcia@ufrgs.br}
\email{wesleyglautenschlaeger@gmail.com}
\email{julianabpedrotti@gmail.com}
\email{thaisa.tamusiunas@gmail.com}
\date{}

\usepackage{graphicx}
\usepackage{amssymb}
\usepackage{amsmath}
\usepackage{amsthm}
\usepackage{indentfirst}
\usepackage{tikz}
\usetikzlibrary{cd}

\newcounter{contador}
\numberwithin{contador}{section}

\newtheorem{theorem}[contador]{Theorem}
\newtheorem{prop}[contador]{Proposition}
\newtheorem{lemma}[contador]{Lemma}
\newtheorem{corollary}[contador]{Corollary}

\theoremstyle{definition}
\newtheorem{defi}[contador]{Definition}

\newtheorem{exe}[contador]{Example}

\newcommand{\G}{\mathcal{G}}
\newcommand{\cH}{\mathcal{H}}
\newcommand{\cK}{\mathcal{K}}
\newcommand{\cP}{\mathcal{P}}
\newcommand{\cS}{\mathcal{S}}
\newcommand{\cA}{\mathcal{A}}

\begin{document}

\maketitle

\begin{abstract}
    We show a classification method for finite groupoids and discuss the cardinality of cosets and its relation with the index.  We prove a generalization of the Lagrange's Theorem and establish a Sylow theory for groupoids.
\end{abstract}

\vspace{0.5 cm}



\section{Introduction}

On group theory, Lagrange's Theorem states that, given a finite group $G$, the order of any subgroup $H$ divides the order of $G$. Precisely, it establishes that the number of cosets of $H$ in $G$ is given by the order of $G$ divided by the order of $H$.

A groupoid is usually presented as a small category whose morphisms are invertible, which is a natural extension of the notion of group. Indeed, any group can be seen as a category with a unique object. An algebraic interpretation of groupoids appeared for the first time in \cite{brandt1927verallgemeinerung}, but a generalization of group theory for the case of groupoids took a while to be studied. A Cayley theorem for groupoids appeared in \cite{ivan2002algebraic}. A theory for normal subgroupoid and quotient groupoid was given in \cite{paques2018galois}. Normal ordered subgroupoids and quotient ordered groupoids were studied in \cite{alyamani2016fibrations}. In \cite{avila2020isomorphism} isomorphism theorems for groupoids were proved, such as results of normal and subnormal groupoid series. In addition, in \cite{avila2020notions} the notions of center, commutator and inner isomorphism for groupoids were presented. 

Our main goal in this paper is to prove a Lagrange's Theorem for groupoids and to show some of its direct consequences in the generalization of group theory. We also extend the Sylow theory, which guarantees the existence of subgroupoids of a given order, as well as some properties about them.

This work is organized as it follows. In section 2 we provide a background about groupoids and fix some notations. In section 3 we determine the order of a finite subgroupoid in terms of its connected components and we present the first part of Lagrange's Theorem. Also, we give a method to classify finite groupoids. In section 4 we discuss about cosets and its relations with the index  and we prove a generalization of the Lagrange's Theorem for groupoids. The last section will address a generalizaton of Sylow theory.

\section{Preliminaries}

Throughout this paper we adopt the algebraic definition of a groupoid, which appears, for instance, in \cite{paques2018galois}. This approach is completely equivalent to its categorical definition. A \textit{groupoid} $\G$ is a nonempty set, equipped with a partially defined binary operation, which we will denote by concatenation, that satisfies the
associative law (whenever it makes sense) and the condition that every element $g \in \G$ has an inverse $g^{-1}$ and a right and a left identity, respectively denoted by $d(g)$ and $r(g)$ and named \textit{domain} and \textit{range} of $g$. It is immediate to check that the composition $gh$ of two elements of $\G$  exists if and only if $d(g) = r(h)$. Also, $\G_0$ will denote the set of identities of $\G$. 

A \emph{subgroupoid} of $\G$ is a nonempty subset $\cH$, equipped with the restriction the of operation of $\G$, that is a groupoid itself. We say that $\cH$ is \emph{wide} if $\cH_0= \G_0$.

Given $e \in \G_0$, we write the set $\G_e = \{ g \in \G : r(g) = d(g) = e \}$. It is easy to verify that $\G_e$ is a group for all $e \in \G_0$, called \textit{the isotropy group associated with} $e$.  The \textit{isotropy subgroupoid of} $\G$ is defined as
\begin{align*}
    \text{Iso}(\G) = \bigcup^\cdot_{e \in \G_0} \G_e.
\end{align*}

A groupoid $\G$ is said to be \emph{connected} if given any $e_1, e_2 \in \G_0$ there exists $g \in \G$ with $d(g) = e_1$ and  $r(g) = e_2$. In a connected groupoid, all the isotropy groups are isomorphic. 

Given $e_1,e_2\in\G_0$ we set $\G(e_1,e_2):=\{g\in \mathcal{G}: d(g)=e_1\,\,\text{and}\,\,r(g)=e_2\}$. It is well-known that any groupoid is a disjoint union of connected subgroupoids. Indeed, we define the following equivalence relation on $\G_0$: for all $e_1,e_2 \in \G_0$, 
\begin{align*}
e_1\sim e_2 &\text{ if and only if } \G(e_1,e_2)\neq \emptyset.
\end{align*}

Every equivalence class $\bar{e} \in\G_0/\!\!\sim$ determines a connected subgroupoid $\G_{\bar{e}}$ of $\G$, whose set of identities is $\bar{e}$. The subgroupoid $\G_{\bar{e}}$ is called the {\it connected component of $\G$ associated to $\bar{e}$}. It is clear that $\G=\dot\cup_{\bar{e}\in \G_0/\!\sim}\G_{\bar{e}}$.

Given a nonempty set $X$, we recall that the \emph{coarse groupoid} associated to $X$ is the groupoid $X^2 = X \times X$, where $\exists(x,y)(u,v)$ if and only if $x = v$ and in this case $(x,y)(u,v) = (u,y)$. The identities of $X^2$ are the elements of the form $(x,x)$ and the inverse element of the pair $(x,y)$ is given by $(x,y)^{-1} = (y, x)$. If $\G$ is a connected groupoid, \cite[Proposition 2.1]{bagio2019partial} relates it closely to its isotropy groups. Under this condition, we have $\G \simeq \G_0^2 \times \G_e$, and
this does not depend on the choice of $e \in \G_0$. Notice that every two finite coarse groupoids with the same amount of identities are isomorphic. Hence we will just denote by $\cA_n$ the coarse groupoid with $n$ identities. Thus, if $\G$ is a finite connected groupoid with $|\G_0| = k$ and $e \in \G_0$, we have $\G \simeq \cA_k \times \G_e$.

\section{Classification of finite groupoids}

Given a finite groupoid, the purpose of this section is to determine the order of a subgroupoid in terms of its connected components and under what conditions this order divides the order of the entire groupoid. This is the first part of Lagrange's Theorem. Also, we want to use this description to classify finite groupoids. We start by studying the case of connected groupoids and then move forward to the general case. 

Fix the notation $|\G|$ for the order of a groupoid $\G$. From now on, $\G$ will always denote a finite groupoid with $k$ identities.

\begin{prop} \label{propdecsubcon}
Let $\cH$ be a connected subgroupoid of a connected groupoid $\G$ and $e \in \cH_0$. If $|\cH_0| = d$, then $\cH \simeq \cA_d \times K$, where $K$ is a subgroup of $\G_e$.
\end{prop}
\begin{proof} It is straightforward.
\end{proof}

We study now the case where $\cH$ is not necessarily connected.

\begin{exe}
Let $\cH$ be a subgroupoid of a connected groupoid $\G$. Suppose that $\cH$ has two connected components $\cK$ and $\mathcal{L}$. Hence $\cK$ and $\mathcal{L}$ are connected subgroupoids of $\G$. Thus, by Proposition \ref{propdecsubcon}, if $d_1 = |\cK_0|$, $d_2 = |\mathcal{L}_0|$, $e \in \cK_0$ and $f \in \mathcal{L}_0$, we obtain $\cK \simeq \cA_{d_1} \times \cK_e$ and $\mathcal{L} \simeq \cA_{d_2} \times \mathcal{L}_f$. Since $\cK_e, \mathcal{L}_f$ are subgroups of $\G_e \simeq \G_f$, it follows that $|\cK_e|$ and $|\mathcal{L}_f|$ divide $|\G_e|$. Hence
\begin{align*}
    |\cH| = |\cK| + |\mathcal{L}| = d_1^2 \cdot |\cK_e| + d_2^2 \cdot |\mathcal{L}_f|.
\end{align*}
\end{exe}

With a simple exercise of computation, the next proposition extends the example above.

\begin{prop} \label{propdecsubnaocon}
Let $\cH$ be a subgroupoid of a connected groupoid $\G$ with connected components $\cK_1, \cK_2, \ldots, \cK_m$ and let $e_i \in (\cK_i)_0$ for all $i \in \{ 1, 2, \ldots , m \}$. Denote by $k_i = |(\cK_i)_0|$. Then 
\begin{align*}
    |\cH| = \sum_{i=1}^m k_i^2 \cdot |(\cK_i)_{e_i}|.
\end{align*}
\end{prop}

The next result is the first part of the Lagrange's Theorem. The complete theorem will be proved in the next section.

\begin{theorem}\label{lagrange1}
Let $\G$ be a groupoid with connected components $\G_i$, $1 \leq i \leq t$, $e_i \in (\G_{i})_0$ and $\cH$ a subgroupoid of $\G$. The following statements hold: \begin{itemize} 
\item[(i)]the order of $\cH$ is of the form
\begin{align*}
    |\cH| =  \sum_{i = 1}^\ell d_i^2m_i,
\end{align*}
where $\ell$ is the amount of connected components of $\cH$, $\sum d_i \leq k$ and $m_i$ divides $|\G_{e_i}|$.
\item[(ii)] if $\G$ and $\cH$ are connected and $|\cH_0|$ divides $|\G_0|$, then $|\cH|$ divides $|\G|$. 
\end{itemize}
\end{theorem}

\begin{proof}
(i) Denote by $\cH_i = \G_i \cap \cH$ and $\cK_{i}^1, \ldots, \cK_i^{n_i}$ the connected components of $\cH_i$. Notice that $\cH_i$ is a subgroupoid of $\G_i$, for $1 \leq i \leq t$. Let $l$ be such that $\cH_i = \G_i \cap \cH \neq \emptyset$, for $1 \leq i \leq l$.  Therefore, by Proposition \ref{propdecsubnaocon},
\begin{align*}
    |\cH_i| = \sum_{j = 1}^{n_i} (k_i^j)^2 \cdot |(\cK_i^j)_{e_i^j}|,
\end{align*}
where $k_i^j = |(\cK_i^j)_0|$, $e_i^j \in (\cK_i^j)_0$  and $|(\cK_i^j)_{e_i^j}|$ divides $\left|\G_{e_i}\right|$, for all $1 \leq i \leq l$ and $1 \leq j \leq n_i$.
Since the connected components of $\G$ are disjoint, it follows that
\begin{align*}
    |\cH| = \sum_{i=1}^l |\cH_i| = \sum_{i = 1}^l \sum_{j=1}^{n_i} (k_i^j)^2 \cdot \left|(\cK_i^j)_{e_i^j}\right|.
\end{align*}

The statement $\sum_{i,j} k_i^j \leq k$ now is obvious since $k_i^j = \left|(\cK_i^j)_0\right|$. In fact, $\cup_{i,j} (\cK_i^j)_0 = \cH_0 \subseteq \G_0$, from where it follows that $\sum_{i,j} k_i^j = \left|\cup_{i,j} (\cK_i^j)_0 \right| \leq |\G_0| = k$.

(ii) Denote by $d = |\cH_0|$. Thus $\cH = \cA_d \times \cH_e$, for any $e \in \cH_0$. Since $|\cH_e|$ divides $|\G_e|$ and $d$ divides $k$, we have that $|\cH| = d^2 \cdot |\cH_e|$ divides $k^2 \cdot |\G_e| = |\G|$.
\end{proof}


The results above show that the order of a finite groupoid $\G$ is of the form
\begin{align*}
    |\G|=n_1^2m_1+\dots+n_\ell^2m_\ell,
\end{align*}
where each $n_i^2m_i$ is the order of a connected component of $\G$ such that $n_i$ is the amount of identities of the $i$-th connected component and $m_i$ is the order of the isotropy group of some identity of the $i$-th connected component. From that we can obtain some interesting results.

\begin{corollary}
Every connected groupoid with order $p_1p_2 \cdots p_n$ is a group, where the $p_i$'s are distinct primes.
\end{corollary}
\begin{proof}
Assume that $\G$ is not a group. Then $p_1p_2 \cdots p_n=k^2m$, where $k \neq 1$, which is a contradiction.
\end{proof}

Given an order $n$, we have a method to find out how many groupoids of order $n$ there are. It is an inductive process, and we use it to classify groupoids using the classification of finite groups and of smaller connected groupoids.

The classification of a groupoid of order $n$ relies on the number of its connected components. In number theory and combinatorics, a partition of $n$ is a way to write it as a sum of positive integers; two sums that differ only in the position of their terms are considered the same partition. A summand of a partition is called a \emph{part}. Using these notations, the problem can be translated as classifying connected groupoids.

Fixed the order $n$ and the number of connected components $t$, we need to write all the partitions \[\sum_{i=1}^{t}a_i = n\] of $n$, where each part is given by $a_i = n_i^2 m_i$, with $n_i, m_i \in \mathbb{N}$ and $\sum n_i = k$.

At last, we must classify each part as a connected groupoid using the decomposition $\cA_{n_i} \times G_i$, where $G_i$ is a group of order $m_i$. We will classify all groupoids with order between 1 and 6 using this method in the next example.

\begin{exe}
If $\G$ is a groupoid of order 1, 2 or 3, then $\G$ is a group or a disjoint union of groups. As all groups of order 1, 2 or 3 are cyclic, the number of groupoids of order 1, 2, 3 are, respectively, 1, 2 and 3.

If $\G$ is a groupoid of order 4, $\G$ is a group - $\mathbb{Z}_4$ or $K$, the Klein group - a disjoint union of groups or $\G = \cA_2$. The groupoid $\cA_2$ is the smallest groupoid which is not a disjoint union of groups. There are 7 groupoids of order 4.

If $|\G| = 5$, the number of partitions of 5 is 7:

\begin{itemize}
    \item For $5$, $\G$ has one connected component, then $\G \simeq \mathbb{Z}_5$;
    \item For $4+1$, we have the union of disjoint connected groupoids, with 3 options for the connected component of order 4; 
    \item For $3+2$, we have the disjoint union of $\mathbb{Z}_3$ and $\mathbb{Z}_2$;
    \item For $3+1+1$, $2+2+1$, $2+1+1+1$ and $1+1+1+1+1$, we have more disjoint unions of groups.
\end{itemize}

Then, we can see that there are 8 groupoids of order 5.

If $\G$ is a groupoid of order 6, $\G$ has at least one connected component and at most six connected components. If $\G$ has only one connected component, $\G$ is a group, since $6= 2 \cdot 3$, two distinct primes. There are two groups of order 6: $\mathbb{Z}_6$, the cyclic group of order 6, and $D_3$, the dihedral group of degree 3.

If $\G$ has two connected components, we have:
\begin{itemize}
    \item For $5 + 1$, we have the disjoint union of $\mathbb{Z}_5$ and the trivial group.
    \item For $4 + 2$, we have the disjoint union of $\mathbb{Z}_2$ and of a connected groupoid of order 4. There are three options: $\mathbb{Z}_4$, $K$ and $\cA_2$.
    \item For $3 + 3$, we have the disjoint union of $\mathbb{Z}_3$ and $\mathbb{Z}_3$.
\end{itemize}

If $\G$ has three connected components, for $4 + 1 + 1$, there are three options. All the other groupoids will have connected components with order less than 4. For groupoids with four, five or six connected components, we have another 6 options, which are the disjoint union of groups. Therefore, there are 16 non-isomorphic groupoids of order 6.
\end{exe}

    
    

\section{Index, Cosets and Lagrange's Theorem}

In this section we will establish relations between the cardinality of the cosets of a finite groupoid and its order, and use it to generalize the Lagrange's Theorem. For that, if $\cH$ is a subgroupoid of $\G$ the relation $\equiv_{\cH}$ in $\G$ is defined in \cite{paques2018galois} as
\begin{align*}
    x \equiv_\cH y \Leftrightarrow \exists  yx^{-1} \text{ and } yx^{-1} \in \cH.
\end{align*}

For $g \in \G$ we define the right coset of $\cH$ in $\G$ that contains $g$ by $$\cH g=\{hg : h\in \cH \textrm{ and } d(h)=r(g) \}.$$ 

We can define left cosets in a similar way. Besides that, there is a bijection between the sets of left and right cosets given by $\cH g \mapsto g^{-1}\cH$. Therefore we can denote the amount of cosets of $\cH$ in $\G$ by $(\G : \cH)$.

In the case of groups, we have that $ (G:H)|H|=|G|$, because every coset has the same cardinality and every element belongs to a coset. This is not the case of groupoids.

\begin{exe}\label{7belo}
Take the groupoid $\G = \cA_3 \times S_2$, where $S_2 = \{1, \sigma\}$ and the subgroupoid $\cH = \{e_1, e_2, g_{21}, g_{12}\}\times S_2 \cup \{e_3\}$ of  $\G$ as in the following simplified diagram:
\begin{figure}[H]
    \centering
    \begin{tikzcd}[column sep = small]
        e_1 \arrow[distance=2em, in=305, out=235]{}{S_2} \arrow[rr] &     & e_2  \\
         & e_3 & 
    \end{tikzcd}
\end{figure}
We can observe that the cosets $g_{23}\cH = \{g_{23}\}$ and $g_{21}\cH = \{ g_{21}, g_{21}\sigma, e_2, $ $ g_{21}g_{12}\sigma \}$ do not have the same cardinality.

Consider $\cK = \{e_1, e_2, g_{21}, g_{12}\}\times S_2$. Notice that $g_{23} \in g_{23}\cK$ while $\cK g_{23}= \emptyset$.
\end{exe}

If $g \in \G$ is such that $g \cH=\emptyset$ or $\cH g=\emptyset$ we will simply write that the coset of $g$ in $\cH$ is empty. In fact, every element belongs to a coset if and only if $\cH$ is a wide subgroupoid of $\G$. 

\begin{lemma} \label{lemma:classe}
Let $g \in \G$. If $r(g) \in \cH$, denote by $\cH^*$ the connected component of $\cH$ such that $r(g) \in \cH^*$. Define
\begin{align*}
    \delta = \begin{cases}
        |\cH_0^*|, & \text{ if } r(g) \in \cH \\
        0, & \text{ if } r(g) \notin \cH.
    \end{cases}
\end{align*}

Then $|\cH g| = \delta \cdot |\cH_{r(g)}|$.
\end{lemma}
\begin{proof}
Notice that
\begin{align*}
    \cH g = \{ hg : h \in \cH, d(h) = r(g) \}.
\end{align*}

Since the cancellation law is valid in the case of groupoids, we have that
\begin{align*}
    |\cH g| & = |\{ hg : h \in \cH, d(h) = r(g) \}| \\ & = |\{ h \in \cH : d(h) = r(g) \}| \\
    & = |\{ h \in \cH^* : d(h) = r(g) \}|.
\end{align*}

Now we can use a simple couting argument and conclude that
\begin{align*}
    |\{ h \in \cH^* : d(h) = r(g) \}| = \delta \cdot |\cH_{r(g)}|.
\end{align*}
\end{proof}

\begin{lemma} \label{leelee}
Let $\cH$ be a subgroupoid of a connected groupoid $G$, $\cH_1$ and $\cH_2$ two connected components of $\cH$, $f \in \G_0 \setminus \cH_0$ and $$\G'=\{g \in \G: r(g) \in (\cH_1)_0 \textrm{ and }  d(g)= f \in (\cH_2)_0 \}.$$

Then the amount of cosets of the type $\cH g$ where $g \in \G'$ is $$|(\cH_2)_0|(\G_e:(\cH_1)_e).$$
\end{lemma}

\begin{proof}
By Proposition \ref{propdecsubcon} we have that $\G \simeq \cA_{k} \times \G_e$ and $\cH_1 \simeq \cA_{|(\cH_1)_0|} \times \cH_{e_1}.$ Thus the elements $g \in \G'$ can be seen as $((f,r(g)),g') \in \cA_{|(\cH_2)_0|} \times \G_e$ where $g' \in \G_e.$ Hence
\begin{align*}
   \cH_1 g & \simeq
   (\cA_{|(\cH_1)_0|} \times \cH_{e_1})((f,r(g)),g') \\ &= \cA_{|(\cH_1)_0|}(f,r(g)) \times \cH_eg'\\
   & =\{((a,a)(f,r(g)),hg'): a \in \cH_1 \textrm{ and } h \in \cH_{e_1}\}.
\end{align*}

Therefore the amount of cosets of the type $\cH g$, for fixed $f \in (\cH_2)_0$, is determined only by the amount of cosets of the isotropy group. Hence, the number of cosets is $(\G_e:(\cH_1)_e).$ Since there are $|(\cH_2)_0|$ distinct identities, the amount of cosets of the type $\cH g$ where $g \in \G'$ is precisely given by $|(\cH_2)_0|(\G_e:(\cH_1)_e)$.
\end{proof}



\begin{theorem}\label{indice} Let $\cH$ be a subgroupoid of a connected groupoid $\G$ with connected components denoted by $\cH_i$ and, for $e_i \in (\cH_i)_0$, the isotropy group $\cH_{e_i}$ of $\cH_i$ for, $1 \leq i \leq n$. Then
\begin{align*}
    (\G:\cH)=|\cH_0|\left(\sum_{i=1}^n (\G_{e} : \cH_{e_i}) \right).
\end{align*}
\end{theorem}

\begin{proof}
The proof will be given by induction in the number of connected components of $\cH$ and in the number of identities of $\cH$. Assume that $\cH$ is connected with only one identity. Given $g \in \G$, the coset $\cH g$ is not empty if and only if there is $h \in \cH$ such that $r(g)=d(h)$. Hence $(\G:\cH)=(\G_e:\cH_e)$.

Now, assume that for any subgroupoid $\cH$ which is connected and has $l$ identities, the equality \[(\G:\cH)=l(\G_e:\cH_e)\] holds.
    
For the induction step, suppose that $\cH$ is connected and has $l+1$ identities denoted by $\{e_1,\dots,e_{l+1}\}$. Consider the subgroupoid $\cH'=\{g \in \cH : d(g) = e_i \textrm{ and } r(g)=e_j, \textrm{ for } 1\leq i,j \leq l \}$. By the induction assumption we have that $(\G:\cH')=l(\G_e:\cH_e)$.

Notice that if $g \in \G$ is such that $d(g) \neq e_i$ and $r(g) \neq e_j$, for $1 \leq i,j \leq l+1$, then the coset of $g$ in $\cH$ is empty. If $d(g) = e_{l+1}$ and $r(g) = e_j$, for $1\leq j \leq l$ or $d(g) = e_{i}$, for $1\leq i \leq l$, and $r(g)= e_{l+1}$, the coset of $g$ in $\cH$ was already counted by $(\G:\cH')$.
Therefore, it only remains to us to observe what happens to the elements $g \in \G$ such that $d(g) = r(g)= e_{l+1}$, that generate $(\G_e:\cH_e)$ additional cosets.
Hence
\begin{align*}
    (\G:\cH)=(\G:\cH')+(\G_e:\cH_e)=l(\G_e:\cH_e)+(\G_e:\cH_e)=(l+1)(\G_e:\cH_e),
\end{align*}
which concludes the induction in the number of identities. 

Assume now that for any subgroupoid $\cH$ which has $m$ connected components and $l$ identities, the equality
\begin{align*}
    (\G:\cH)=|\cH_0|\left( \sum_{i=1}^m (\G_{e}: \cH_{e_i})  \right) 
\end{align*} holds.
Suppose that $\cH$ has $m+1$ connected components and $l$ identities. By induction assumption, the subgroupoid $\cH' = \cup_{i=1}^m \cH_i$ is such that 
\begin{align*}
    (\G:\cH')=|\cH'_0|\left(\sum_{i=1}^m (\G_{e}: \cH_{e_i})  \right)
\end{align*}
which is the number of cosets of elements $g \in \G$ such that $d(g)$ and $r(g)$ are in $\cH'_0$. 

Notice that $\cH_{m+1}$ is a connected component, so the first part of the proof says that the cosets of $g \in \G$ such that $d(g) \in \cH_{m+1}$ and  $r(g) \in \cH_{m+1}$ are counted as
\begin{align*}
    |(\cH_{m+1})_0|(\G_e : \cH_{e_{m+1}}).
\end{align*}
If $g \in \G$ is such that $d(g) \in \cH_{m+1}$ and $r(g) \in \cH_{i}$, for some $1 \leq i \leq m$, the number of cosets of the type $\cH g$ is $|(\cH_{m+1})_0|(\G_e:\cH_{e_i})|$ by Lemma \ref{leelee}. Since there are $m$ connected components we have that the number of cosets generated by $\{g \in \G:d(g) \in \cH_{m+1} \textrm{ and }r(g) \in \cH_{i}$\} is
\begin{align*}
    |(\cH_{m+1})_0|\left(\sum_{i=1}^m (\G_{e}:\cH_{e_i})\right).
\end{align*}
If $g \in \G$ is such that $d(g) \in \cH_{i}$, for some $1 \leq i \leq m$ and $r(g) \in  \cH_{m+1}$, the number of cosets of the type $\cH g$ is $(\G_{e}:\cH_{e_{m+1}})$ by Lemma \ref{leelee}. Since there are $m$ connected components we have that the number of cosets generated by $\{g \in \G:d(g) \in \cH_{i} \textrm{ and }r(g) \in \cH_{m+1}$\} is $|\cH'_0|(\G_{e}:\cH_{e_{m+1}})$.
Hence
\begin{align*}
    (\G:\cH)&=|\cH'_0|\left(\sum_{i=1}^m (\G_{e}: \cH_{e_i})  \right)+|(\cH_{m+1})_0|\left(\sum_{i=1}^m (\G_{e}:\cH_{e_i})\right)\\
     & + |\cH'_0|(\G_{e}:\cH_{e_{m+1}})+
    |(\cH_{m+1})_{0}|(\G_{e}:\cH_{e_{m+1}})  \\
    &=|\cH_0|\left(\sum_{i=1}^m (\G_{e}: \cH_{e_i})  \right)+
    |\cH_{0}|(\G_{e}: \cH_{e_{m+1}})\\
    &=|\cH_0|\left(\sum_{i=1}^{m+1} (\G_{e}: \cH_{e_i}) \right).
\end{align*}
\end{proof}

\begin{theorem} Let $\G_j$ be the connect components of $\G$ and $\G_{e_j}$ the isotropy group of $\G_j$ for $1 \leq j \leq t$, and let $\cH$ be a subgroupoid of $\G$ with connected components denoted by $\cK_i$ and, for $e_i \in (\cK_i)_0$, $\cK_{e_i}$ the isotropy group of $\cK_i$, for $1 \leq i \leq n$ and $\cH_j=\G_j \cap \cH$. Then
\begin{align*}
    (\G:\cH)=\sum_{j=1}^t |(\cH_j)_0|\left(\sum_{i=1}^n (\G_{e_i} : \cK_{e_i})\right).
\end{align*}
\end{theorem}
\begin{proof}
Just apply the Theorem \ref{indice} in each connected component of $\G$.
\end{proof}

\begin{theorem}\label{resp} Let $\G$ be a connected groupoid and $\cH$ a wide subgroupoid with connected components $\cH_i$ and, for $e_i \in (\cH_i)_0$, $\cH_{e_i}$ the isotropy group of $\cH_i$ for $1 \leq i \leq n$. Then
\begin{align*}
    |\G| = |\G_0|\left(\sum_{i=1}^n (\G_{e_i} : \cH_{e_i})|(\cH_i)_{0}||\cH_{e_i}| \right).
\end{align*}
\end{theorem}
\begin{proof} We have that
\begin{align*}
    |\G| & = |\G_0|^2|\G_e| \\
    &=|\G_0|\left(\sum_{i=1}^n|(\cH_{i})_0|\right)|\G_e|\\
    &=|\G_0|\left(\sum_{i=1}^n|(\cH_{i})_0||\G_e|\right)\\
    &=|\G_0|\left(\sum_{i=1}^n|(\cH_{i})_0||\G_{e_i}|\right)\\
    &=  |\G_0| \left(\sum_{i=1}^n|(\cH_{i})_0| (\G_{e} : \cH_{e_i})|\cH_{e_i}|\right),
\end{align*}which concludes the proof.
\end{proof}

Notice that if $H$ is a subgroup of a finite group $G$, Theorem \ref{resp} precisely states that $|G| = (G:H)|H|.$ The condition of $\G$ being a connected groupoid can be disregarded, as we can see below.

The next result is the second part of the Lagrange's Theorem.

\begin{theorem}\label{lagrange2} Let $\G_j$ be the connected components of $\G$ and $\G_{e_j}$ the isotropy group of $\G_j$ for $1 \leq j \leq t$, let $\cH$ be a wide subgroupoid of $\G$ with connected components $\cK_i$ and, for $e_i \in (\cK_i)_0$, let $\cK_{e_i}$ be the isotropy group of $\cK_i$, for $1 \leq i \leq n$ and $\cH_j=\G_j \cap \cH$. Then
\end{theorem}
\begin{align*}
   |\G| = \sum_{j=1}^t |(\G_j)_0|\left(\sum_{i=1}^n (\G_{e_j} : \cK_{e_i})|(\cH_j)_{0}||\cK_{e_i}| \right).
\end{align*}
\begin{proof}
Just apply the Theorem \ref{resp} in each connected component of $\G$.\end{proof}

Using Theorem \ref{lagrange1} and Theorem \ref{lagrange2}, we can finally state a generalization of the Lagrange's Theorem, without concern about the connectedness of the groupoid.

\begin{theorem}[Lagrange's Theorem for Groupoids]\label{lagrange}
Let $\G$ be a groupoid with connected components $\G_j$, $1 \leq j \leq t$, $e_j \in (\G_{j})_0$ and $\cH$ a subgroupoid of $\G$. The following statements hold: \begin{itemize} 
\item[(i)]the order of $\cH$ is of the form
\begin{align*}
    |\cH| =  \sum_{j = 1}^\ell d_j^2m_j,
\end{align*}
where $\ell$ is the amount of connected components of $\cH$, $\sum d_j \leq k$ and $m_j$ divides $|\G_{e_j}|$.
\item[(ii)] if $\G$ and $\cH$ are connected and $|\cH_0|$ divides $|\G_0|$, then $|\cH|$ divides $|\G|$.
\item[(iii)] if $\G_{e_j}$ is the isotropy group of $\G_j$, $\cH$ is wide with connected components $\cK_i$ and, for $e_i \in (\cK_i)_0$, $\cK_{e_i}$ is the isotropy group of $\cK_i$, for $1 \leq i \leq n$ and $\cH_j=\G_j \cap \cH$, then
\begin{align*}
   |\G| = \sum_{j=1}^t |(\G_j)_0|\left(\sum_{i=1}^n (\G_{e_j} : \cK_{e_i})|(\cH_j)_{0}||\cK_{e_i}| \right).
\end{align*}
\end{itemize}
\end{theorem}

\begin{corollary}
Let $\G$ be a connected groupoid and $\cH$ be a wide subgroupoid with connected components $\cH_i$ and isotropy groups $\cH_{e_i}$, $e_i \in (\cH_i)_0$, such that $|(\cH_i)_0|$ and $|\cH_{e_i}|$ are the same for all $1 \leq i \leq n$. Then
\begin{align*}
    |\G| = (\G : \cH)|(\cH_i)_{0}||\cH_{e_i}|.
\end{align*}
\end{corollary}
\begin{proof}
Using the Theorems \ref{indice} and \ref{resp}, we obtain
\begin{align*}
    |\G|&=|\G_0|\left(\sum_{i=1}^n (\G_{e_i} : \cH_{e_i})|(\cH_{i})_{0}||\cH_{e_i}|\right)\\
    &=
    \left(|\G_0|\sum_{i=1}^n (\G_{e_i} : \cH_{e_i})\right)|(\cH_{i})_{0}||\cH_{e_i}|\\
    &=(\G : \cH)|(\cH_i)_{0}||\cH_{e_i}|.
\end{align*}
\end{proof}

\section{Sylow Theorems}

The purpose of this section is to answer the following question: given a connected groupoid $\G$ with $|\G_0| = k$ and $d_1, d_2, \ldots, d_\ell, m_1, m_2, \ldots, m_\ell \in \mathbb{N}$ with $\sum_i d_i \leq k$ and $m_i$ a divisor of $|\G_e|$ for all $i$, when is there a subgroupoid $\cH$ of $\G$ with order $d_1^2m_1 + d_2^2m_2 + \cdots + d_\ell^2m_\ell$?

During this section we will consider $\G$ a connected groupoid. The disconnected case can be obtained by applying the results of this section to each connected component of the groupoid. Our first result is a consequence of the First Sylow Theorem for groups.

\begin{lemma} \label{lemasylow1}
Consider $e \in \G_0$ and $p$ a prime such that $|\G_e| = p^mb$, with \emph{gcd}$(p,b) = 1$. Then for all $0 \leq n \leq m$ there is a wide connected subgroupoid $\cH$ of $\G$ such that $|\cH| = k^2p^n$. In particular, $|\cH|$ divides $|\G|$.

On the other hand, if $1 \leq d \leq k$, there is a connected subgroupoid $\cH$ of $\G$ such that $|\cH_0| = d$ and $|\cH|= d^2p^n$, for all $0 \leq n \leq m$.
\end{lemma}
\begin{proof}
Since $\G$ is connected, $\G \simeq \cA_k \times \G_e$. We know that $\G_e$ is a group, so we can use the First Sylow Theorem for groups to obtain a subgroup $K$ of $\G_e$ with order $p^n$ for all $0 \leq n \leq m$. 

Consider $\cH = \cA_k \times K$. It is evident that $\cH$ is a wide connected subgroupoid of $\G$. Therefore $|\cH|$ divides $|\G|$ by the Theorem \ref{lagrange}. Besides that, $\cH_e$ is precisely $K$, since $K \subseteq \G_e$. Thus $|\cH_e| = |K| = p^n$.

For the second statement, consider a connected subgroupoid $\G' = \cA_d \times \G_e$ of $\G$ and use the argument above to obtain a wide connected subgroupoid $\cH$ of $\G'$. The result is now direct.
\end{proof}

Once $d$ and $p$ above are fixed, inspired by the group notation, we will call a subgroupoid $\cH$ obtained by $(d,p)$-subgroupoid of $\G$. When $|\cH_e| = p^m$, we denote by $(d,p)$-Sylow subgroupoid. If $d = k$, we will write only $p$-subgroupoid or $p$-Sylow subgroupoid.

The Second Sylow Theorem for groups tells us how the $p$-Sylow subgroups are related. We will now generalize this result for the case of connected groupoids. Before that, we will recall the definition of normal subgroupoid and define characteristic subgroupoid.

\begin{defi}
Let $\G$ be a (not necessarily finite or connected) groupoid and $H$ a wide subgroupoid of $\G$. \begin{itemize} \item[(i)] We say that $\cH$ is \emph{normal} and denote $\cH \triangleleft \G$ if $g^{-1}\cH g \subseteq \cH$ for all $g \in G$. 
\item[(ii)] Define $\cA(\G) = \{ f : \G_e \to \G_{e'} : e,e' \in \G_0 \text{ and } f \text{ is an isomorphism} \}.$ We say that $\cH$ is \emph{characteristic} and denote $\cH \blacktriangleleft \G$ if $\cH$ is invariant under all elements of $\cA(\G)$. That is, if $f(\cH \cap \text{dom}(f)) = \cH \cap \text{Im}(f)$.\end{itemize}
\end{defi}

It is clear that every characteristic subgroupoid is normal. In fact, a subgroupoid is normal if and only if it is invariant under the inner isomorphisms of $\G$, that together are a subset of $\cA(\G)$ \cite[Proposition 5.2]{avila2020notions}. Notice that $\cA(\G) \neq \text{Aut}(\G)$ in general.

\begin{exe}
It is evident that $\G_0 \blacktriangleleft \G$ and $\G \blacktriangleleft \G$. Defining  \[Z(\G) = \{ g \in \text{Iso}(\G) : gh = hg \text{ for all } h \in \G_{d(g)} \},\]the center of $\G$ \cite[Definition 4.1]{avila2020notions}, we have that $Z(\G) \blacktriangleleft \G$. 

In fact, take $f \in \cA(\G)$. Consider $f : \G_e \to \G_{e'}$. Then $Z(\G) \cap \G_e = Z(\G_e)$. Hence, $f(Z(\G) \cap \G_e) = f(Z(\G_e))$. We will show that \[f(Z(\G_e)) = Z(\G_{e'}) = Z(\G) \cap \G_{e'}.\]

Let $g' \in f(Z(\G_e))$. Thus there is $g \in Z(\G_e)$ such that $g' = f(g)$ and $r(g') = e' = d(g')$. Let $h' \in \G_{e'}$. Since $f$ is an isomorphism, there is $h \in \G_e$ such that $f(h') = h$. Hence
\begin{align*}
    g'h' = f(g)f(h) = f(gh) = f(hg) = f(h)f(g) = h'g',
\end{align*}
because $g \in Z(\G_e)$. Thus $g' \in \G_{e'}$ as we wanted. The other inclusion is given similarly since $f$ is an isomorphism.
\end{exe}

A classic result in group theory is that the relation $\triangleleft$ is not transitive, but can be repaired with the relation $\blacktriangleleft$. This result is also true in groupoid theory.

\begin{prop}
Let $\G$ be a groupoid and let $\cH,\cK$ be subgroupoids of $\G$ such that $\cK \blacktriangleleft \cH \triangleleft \G$. Then $\cK \triangleleft \G$.
\end{prop}
\begin{proof}
Consider $I_g : \G_{d(g)} \to \G_{r(g)}$ the partial inner isomorphism of $\G$ given by $I_g(x) = gxg^{-1}$. Since $\cH$ is normal, we have that \[I_g(\cH_{d(g)}) = I_g(\cH \cap \G_{d(g)}) = \cH \cap \G_{r(g)} = \cH_{r(g)}.\]

Hence $I_g|_\cH : \cH_{d(g)} \to \cH_{r(g)}$ is an element of $\cA(\G)$. Since $\cK$ is characteristic, 
\begin{align*}
    I_g(\cK \cap \G_{d(g)}) & = I_g(\cK \cap \cH_{d(g)}) \\
    & = I_g|_\cH(\cK \cap \cH_{d(g)}) \\ 
    & = \cK \cap \cH_{r(g)} = \cK \cap \G_{r(g)}.
\end{align*}

This shows us that $\cK$ is invariant under every partial inner isomorphism of $\G$, that is, $\cK$ is normal.
\end{proof}

\begin{defi}
Denote by $\mathcal{W}_C(\G)$ the set of all wide connected subgroupoids of $\G$ and let $\cH \in \mathcal{W}_C(\G)$. Define
\begin{align*}
    \mathcal{I}_g : \mathcal{W}_C(\G) & \to \mathcal{W}_C(\G) \\
    \cH \simeq \cA_k \times \cH_{d(g)} & \mapsto \cA_k \times g\cH_{d(g)}g^{-1}.
\end{align*}

We will call the wide connected subgroupoid $\mathcal{I}_g(\cH)$ by \emph{isotropic conjugate} of $\cH$. 
\end{defi}

\begin{lemma} \label{lemasylow2}
Let $p$ be a prime, $1 \leq d \leq k$ and $n_{d,p}$ the number of $(d,p)$-Sylow subgroupoids of $\G$.
\begin{enumerate}
    \item[(i)] All $(d,p)$-Sylow subgroupoids of $\G$ are isotropic conjugates. In particular, a $p$-Sylow subgroupoid $\cH$ is such that $\cH \triangleleft \G$ if and only if $n_{k,p} = 1$. In this case, $\cH \blacktriangleleft \G$.
    
    \item[(ii)] If $\cP$ is a $(d,p)$-subgroupoid of $\G$, there is a $(d,p)$-Sylow subgroupoid $\cS$ of $\G$ such that $\cP \subseteq \cS$.
    
    \item[(iii)] If $\cS$ is a $p$-Sylow subgroupoid, we have that $n_{k,p} = (\G_e : N_{\G_e}(\cS_e))$, for any $e \in \G_0$.
\end{enumerate}
\end{lemma}
\begin{proof}

(i): The statement is equivalent to show that given $\cH, \cK$ $(d,p)$-Sylow subgroupoids of $\G$, $e \in \cH_0$ and $f \in \cK_0$, we have that there is $g \in \G$ such that $g\cH_eg^{-1} = \cK_f$. 

Since $\G$ is connected, there is $x \in \G$ with $r(x) = f$, $d(x) = e$. We already know that $\G_e \simeq \G_f$ via $I_x$. Thus $I_x(\cH_e) = x\cH_ex^{-1}$ is a $p$-Sylow subgroup of $\G_f$. Since all $p$-Sylow subgroups of a group are conjugates, there is $y \in \G_f$ such that $yx\cH_ex^{-1}y^{-1} = \cK_f$. Therefore $g = yx$ is the element that we wanted. 

On the other hand, let $\cH$ be a $(d,p)$-Sylow subgroupoid. It is obvious that $I_g(\cH_e) = g^{-1}\cH_eg$ is a $p$-Sylow subgroup of $\G_f$. Therefore we can construct a suitable $(d,p)$-Sylow subgroupoid $\cK = \cA_d \times g^{-1}\cH_eg$ such that $\cK_f = g^{-1}\cH_eg$.

Now, notice that $\cH \triangleleft \G$ if and only if $g^{-1}\cH g \subseteq \cH$, for all $g \in G$. We have that $g^{-1}\cH g = g^{-1}\cH_{r(g)}g$ is precisely an isotropy group for some $p$-Sylow subgroupoid of $\G$. Since $\cH$ is unique, we have that $g^{-1}\cH_{r(g)}g = \cH_{d(g)} \subseteq \cH$ proving that $\cH$ is normal. For the converse, notice that if $\cH$ is normal then $g^{-1}\cH_{r(g)}g = \cH_{d(g)}$ for all $g \in \G$. But this implies that every $p$-Sylow subgroupoid of $\G$ has the exact same elements as $\cH$. Thus, $\cH$ must be unique.

In fact, we have proved even more: a $p$-Sylow subgroupoid $\cH$ is normal in $\G$ if and only if every $p$-Sylow subgrpup $\cH_e$ is normal in $\G_e$ for all $e \in \G_0$ if and only if every $p$-Sylow subgroup $\cH_e$ is characteristic in $\G_e$ for all $e \in \G_0$. The last equivalence follows from the Second Sylow Theorem for groups. Since the definition of characteristic subgroupoid depends only on the isotropy subgroups, it follows that $\cH \blacktriangleleft \G$.

(ii): If $\cP$ is a $(d,p)$-subgroupoid of $\G$, then $\cP_e$ is a $p$-subgroup of $\G_e$, for all $e \in \cP_0$. Hence there is $\cS_e \subseteq \G_e$ $p$-Sylow subgroup for all $e \in \cP_0$. Now take $\cS = \cA_d \times \cS_e$.

(iii): By the Second Sylow Theorem for groups we have that $\G_e$ has $(\G_e : N_{\G_e}(\cS_e))$ $p$-Sylow subgroups. The result follows directly from (i).
\end{proof}

\begin{exe} \label{exe:subgrupoides}
Let $\G= \cA_3 \times D_3$, where $D_3 = \{1, \rho, \rho^2, \tau_1, \tau_2, \tau_3\}$ is the dihedral group of degree 3. We know that $|\G| = 3^2 \cdot 6 = 54$. There are three 2-Sylow subgroups given by $\langle \tau_i \rangle$, $1 \leq i \leq 3,$ and only one 3-Sylow subgroup given by $\langle \rho \rangle$. Since $|\G_0| = 3$, we have the options $d=1$, $d=2$ or $d=3$.

For $d=1$, take $\cA_1 \times H \simeq H$, where $H$ is a Sylow subgroup of $D_3$. We have that $n_{1,3} = 3 \cdot 1$ and $n_{1,2} = 3 \cdot 3$, given the choice of the identity $e_i \in G$ and of the Sylow subgroup of $D_3$. The $(1,3)$-Sylow subgroupoids of $G$ are
\[\{e_1\}\times \langle \rho \rangle, \{e_2\}\times \langle \rho \rangle \textrm{ and } \{e_3\}\times \langle \rho \rangle.\]

For $d=2$, taking $\cA_2 \times H$ we obtain $n_{2,3} = 3 \cdot 1 = 3$ and $n_{2,2} = 3 \cdot 3$, since we have three distinct but isomorphic coarse groupoids $\cA_2$ in $\cA_3$. Denoting by $\cA_{2}^{ij}$ the coarse subgroupoid of $\cA_3$ that contains $e_i, e_j$, it follows that the three $(2,3)$-Sylow subgroupoids of $\G$ are isotropic conjugates of the form
\[\cA_{2}^{ij} \times \langle \rho \rangle,\] and have order 12. The nine $(2,2)$-Sylow subgroupoids, of the form $\cA_{2}^{ij} \times \langle \tau_k \rangle$ have order 8. Observe that in this case the orders of the subgroupoids do not divide the order of $G$.

Taking $d=3$, we have the wide Sylow subgroupoids of $\G$. Thus, $n_{3,3} = 1$ and $n_{3,2} = 3$. Now observe that the $(3,3)$-Sylow subgroupoid $\cA_3 \times \langle \rho \rangle$ is unique and therefore normal. On the other hand, the $(2,3)$-Sylow subgroupoids $\cA_3 \times \langle \tau_i \rangle$ are not normal since they are isotropic conjugated.
\end{exe}

The next lemma will give us more information about the number $n_{d,p}$.

\begin{lemma} \label{lemasylow3}
Let $p$ be a prime and let $e \in \G_0$ be such that $|\G_e| = p^mb$ with \emph{gcd}$(p,b)=1$. Then 
\[n_{d,p} = N\cdot \binom{k}{d},\] where $N$ is such that
\[\begin{cases}
    N \textrm{ divides } b, \\
    N \equiv 1 \mod p.
\end{cases}\]
\end{lemma}
\begin{proof}
Consider $N=n_p$ the number of $p$-Sylow subgroups of $\G_e$. By the Third Sylow Theorem for groups, $N | b$ and $N\equiv 1 \mod p$.

Notice that the Sylow subgroupoids of $\G$ are of the form $\cA_d \times H$ where $H$ is a $p$-Sylow subgroup of $\G_e$ and $\cA_d \subset \cA_k$. By a counting exercise we can see that there are $\binom{k}{d}$ isomorphic copies of $\cA_d$.
\end{proof}

We can now state the Sylow theorems for connected groupoids.

\begin{theorem}[First Sylow Theorem]
Let $p_1, p_2, \ldots , p_\ell$ be primes such that $|\G_e| = p_i^{m_i}b_i$, with \emph{gcd}$(p_i,b_i) = 1$, for all $1 \leq i \leq \ell$. Then for all $0 \leq n_i \leq m_i$ and $1 \leq d_1, d_2, \ldots, d_\ell \leq k$ with $d = \sum_i d_i \leq k$ there is a subgroupoid $\cH$ of $\G$ such that $|\cH_0| = d$ and $$|\cH|= \sum_{i=1}^{\ell} d_i^2p_i^{n_i}.$$
\end{theorem}
\begin{proof}
Just take $\cH_i$ as the connected $(d_i,p_i)$-subgroupoid of $\G$ as in Lemma \ref{lemasylow1} for all $1 \leq i \leq \ell$. Defining $\cH = \bigcup\limits^\cdot_i \cH_i$ we obtain the result.
\end{proof}

From now on, fix $D = (d_1, \ldots, d_\ell), P = (p_1, \ldots, p_\ell) \in \mathbb{N}^{\ell}$ where $\sum d_i \leq k$ and each $p_i$ is a prime as in the previous lemma. The subgroupoid above will be denoted by $(D,P)$-subgroupoid of $\G$. If $n_i = m_i$ for all $1 \leq i \leq \ell$, we will write $(D,P)$-Sylow subgroupoid. 

\begin{defi}
Let $\cH$ be a wide $(D,P)$-Sylow subgroupoid of $\G$, that is, $k = \sum_i d_i$. Consider $\cH_i$ the $(d_i,p_i)$-Sylow subgroupoid of $\G$ that is a connected component of $\cH$. A \emph{connected components permutation} of $\cH$ is a subgroupoid $\cK$ of $\G$ such that $\cK \simeq \cH$ and, denoting by $\cK_i$ the connected component isomorphic to $\cH_i$ for all $1 \leq i \leq \ell$, there is at least one $i$ such that $(\cK_i)_0 \neq (\cH_i)_0$.
\end{defi}

\begin{exe}
As in Example \ref{exe:subgrupoides}, let $G=\cA_3\times D_3$. Consider the disjoint union of the subgroupoids $\cH=\{e_1\} \times \langle \rho \rangle$ and $\cK=\{e_2, e_3, g_{32}, g_{23}\} \times \langle \tau_1 \rangle$.

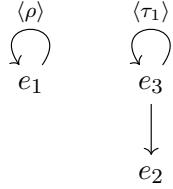
\begin{figure}[H]
    \centering
    \begin{tikzcd}
        e_1 \arrow[distance=2em, in=125, out=55]{}[above]{\langle \rho \rangle} & e_3 \arrow[distance=2em, in=125, out=55]{}[above]{\langle \tau_1 \rangle} \arrow[d] \\
         & e_2 
    \end{tikzcd}
\end{figure}

A connected components permutation of $\cH \cup \cK$ is the disjoint union $\cH' \cup \cK'$, where $\cH'= \{e_3\} \times \langle \rho \rangle \simeq \cH$ and $\cK'=\{e_1, e_2, g_{12}, g_{21}\} \times \langle \tau_1 \rangle \simeq K$.
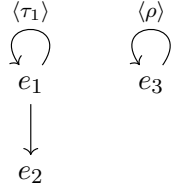
\begin{figure}[H]
    \centering
    \begin{tikzcd}
        e_1 \arrow{d} \arrow[distance=2em, in=125, out=55]{}[above]{\langle \tau_1 \rangle}  & e_3 \arrow[distance=2em, in=125, out=55]{}[above]{\langle \rho \rangle}  \\
        e_2 & 
    \end{tikzcd}
\end{figure}

Let $\cK'' = \{e_1, e_2, g_{12}, g_{21}\} \times \langle \tau_2 \rangle \simeq \cK \simeq \cK'$. We will \emph{not} consider the subgroupoid $\cH' \cup \cK''$
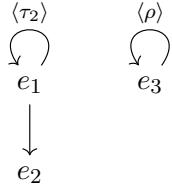
\begin{figure}[H]
    \centering
    \begin{tikzcd}
        e_1 \arrow{d} \arrow[distance=2em, in=125, out=55]{}[above]{\langle \tau_2 \rangle}  & e_3 \arrow[distance=2em, in=125, out=55]{}[above]{\langle \rho \rangle}  \\
        e_2 & 
    \end{tikzcd}
\end{figure}
\noindent as a connected components permutation of $\cH \cup \cK$, even that $\cH \cup \cK \simeq \cH' \cup \cK''$, because the isotropy group is not the same.
\end{exe}

The following lemma is a counting exercise.

\begin{lemma}
Given $\cH$ a wide $(D,P)$-Sylow subgroupoid of $\G$, there are
\begin{align*}
    \binom{k}{D} := \frac{k!}{d_1!d_2! \cdots d_\ell!}
\end{align*}
connected components permutations of $\cH$. 
\end{lemma}

With this notation we can state the Second Sylow Theorem.

\begin{theorem}[Second Sylow Theorem]
Let $p_1, \ldots, p_\ell$ be any primes, $1 \leq d_1, \ldots, d_\ell \leq k$ and $n_{D,P}$ be the number of $(D,P)$-Sylow subgroupoids of $\G$.
\begin{enumerate}
    \item[(i)] Every two $(D,P)$-Sylow subgroupoids of $\G$ are isotropic conjugates besides connected components permutations. In particular, a $(D,P)$-Sylow subgroupoid $\cH$ of $\G$ is normal if and only if $n_{D,P} = \binom{k}{D}$. In this case $\cH$ and all of its connected components permutations are characteristic.
    
    \item[(ii)] If $\cP$ is a $(D,P)$-subgroupoid of $\G$, there is a $(D,P)$-Sylow subgroupoid of $\G$ such that $\cP \subseteq \cS$.
    
    \item[(iii)] If $\cS$ is a $(D,P)$-Sylow subgroupoid, we have that \[n_{D,P} = \binom{k}{D}\prod_i (\G_{e_i} : N_{\G_{e_i}}(\cS_{e_i})),\]
    where the $\cS_{e_i}$ are the isotropy groups of each connected component $\cS_i$ of $\cS$ for some $e_i \in (\cS_i)_0$.
\end{enumerate}
\end{theorem}

\begin{proof}
(i): Let $\cH_1, \ldots, \cH_\ell$ be the connected components of $\cH$. Define, for all $1 \leq i \leq \ell$, $\G_i = \cA_{d_i} \times \G_e$, where $e \in (\cH_i)_0$ and $\cA_{d_i} \subseteq \cA_k$ is the coarse groupoid whose identities are exactly the same as those of $\cH_i$. We have that $\cH_i$ is a connected $p_i$-Sylow subgroupoid of $\G_i$. Hence we can use Lemma \ref{lemasylow2} to obtain that every other connected $p_i$-Sylow subgroupoid of $\G_i$ is an isotropic conjugate of $\cH_i$. Thefore all $(D,P)$-Sylow subgroupoids of $\G$ are isotropic conjugates besides connected components permutations. 

Notice that when $\cH$ is normal, each $\cH_e$ is normal in $\G_e$. But that is the same as saying that, once a connected components permutation is fixed, there is an unique $(D,P)$-Sylow subgroupoid of $\G$. Hence, if $n_{D,P}$ is exactly the number of connected components permutations of $\cH$, every connected components permutations of $\cH$ is normal in $\G$. The statement about characteristic subgroupoids follows directly.

(ii) is evident. We will now prove (iii). Notice that
\begin{align*}
    n_{D,P} = \binom{k}{D}\prod_{i=1}^\ell n_{d_i,p_i},
\end{align*}
where $n_{d_i,p_i} = (\G_{e_i} : N_{\G_{e_i}}(\cS_{e_i}))$ is the number of connected $p_i$-Sylow subgroupoids of $\G_i$.
\end{proof}

Finally, the Third Sylow Theorem gives us a system of congruences involving $ n_{D,P}$. 

\begin{theorem}[Third Sylow Theorem] \label{teosylow3}
Let $p_1, \ldots, p_\ell$ be primes and let $e \in \G_0$ be such that $|\G_{e}| = p_i^{m_i}b_i$ with \emph{gcd}$(p_i,b_i)=1$. Then 
\[n_{D,P} = \left ( \prod_{i=1}^\ell N_i \right ) \cdot \binom{k}{D},\] where each $N_i$ is such that
\[\begin{cases}
    N_i \textrm{ divides } b_i, \\
    N_i \equiv 1 \mod p_i.
\end{cases}\]
\end{theorem}
\begin{proof}
By the Second Sylow Theorem (iii) it follows that
\begin{align*}
    n_{D,P} = \binom{k}{D}\prod_{i=1}^\ell n_{d_i,p_i},
\end{align*}
where $n_{d_i,p_i}$ is the number of connected $p_i$-Sylow subgroupoids of $\G_i$. By Lemma \ref{lemasylow3},
\begin{align*}
    n_{d_i,p_i} = N_i \cdot \binom{d_i}{d_i} = N_i,
\end{align*}
where
\[\begin{cases}
    N_i \textrm{ divides } b_i, \\
    N_i \equiv 1 \mod p_i.
\end{cases}\]
\end{proof}

We will conclude this work by applying the results above in the groupoid $\G = \cA_7 \times \G_e$ where $|\G_e| = 105$.

\begin{exe}
Consider $k = 7$, $D = (1,3,3)$ and $P = (3,5,7)$. If $\G$ is a connected groupoid with $\G \simeq \cA_7 \times \G_e$ where $|\G_e| = 3 \cdot 5 \cdot 7 = 105$, we want to know the possible values of $n_{D,P}$. We already know that
\begin{align*}
n_{D,P} = \binom{7}{1,3,3}\prod_{i=1}^3 N_i,
\end{align*}
where $N_1$ divides 35, $N_2$ divides 21 and $N_3$ divides 15. So
\begin{align*}
    N_1 \in \{ 1, 5, 7, 35 \}, \\
    N_2 \in \{ 1, 3, 7, 21 \}, \\
    N_3 \in \{ 1, 3, 5, 15 \}.
\end{align*}

Hence
\begin{align*}
    n_{D,P} = N_1 \cdot N_2 \cdot N_3 \cdot 140,
\end{align*}
where $N_1 = 1$ or $N_1 = 7$, $N_2 = 1$ or $N_2=21$ and $N_3 = 1$ or $N_3=15$, because the congruences on Theorem \ref{teosylow3} hold. For more information, we would need to know more about the group $\G_e$. For example, there are two non-isomorphic groups of order 105. 

If $\G_e = \mathbb{Z}_{105}$, we have that $N_1 = N_2 = N_3 = 1$, since $\mathbb{Z}_{105}$ is abelian and all of its subgroups are normal. In this case, $\G$ is abelian in the sense of \cite{avila2020isomorphism}. So $n_{D,P} = 140$ and all the 140 $(D,P)$-Sylow subgroupoids of $\G$ are normal in $\G$.

Denoting by $H_p$ a $p$-subgroup of $\G_e$, for $p \in \{ 3, 5, 7 \}$, we have that $n_5 = 21$ and $n_7 = 15$ cannot happen at the same time. In fact, we would have more than 105 elements. So we have that $n_5 = 1$ or $n_7 = 1$. In either case, $H = H_5H_7$ is normal since its index is 3. Besides that, $|H| = 5 \times 7 = 35$. Since there is only one group of order 35, we have that $H \simeq \mathbb{Z}_{35}$.

Define a representation of some $H_3 \simeq \mathbb{Z}_3 \simeq \langle h : h^3 = e \rangle$ on $H$ by
\begin{align*}
    \varphi : H_3 \times \text{Aut}(H) & \to \text{Aut}(H) \\
    (h,\theta) & \mapsto h^{-1} \theta h.
\end{align*}

This will guarantee us that the semidirect product $G = H \rtimes_{\varphi} H_3 \simeq \mathbb{Z}_{35} \rtimes \mathbb{Z}_{3}$ is well-defined and $|G| = 105$, so $\G_e = G$.

Actually, we can prove that $n_5 = n_7 = 1$, so $N_2 = N_3 = 1$, and that $n_3 = 7$, so $N_1 = 7$. Then, we would have $n_{D,P} = 7 \cdot 140 = 980$ and these subgroupoids would not be normal in $\G$.

\end{exe}

\bibliographystyle{amsalpha}

\begin{thebibliography}{99}


\bibitem{alyamani2016fibrations} N. Alyamani; N. D. Gilbert; E. C. Miller, \textit{Fibrations of ordered groupoids and the factorization of ordered functors}, Appl. Categor. Struct., 24 (2016), 121-146.

\bibitem{avila2020isomorphism} J. Ávila; V. Marín; H. Pinedo, \textit{Isomorphism Theorems for Groupoids and Some Applications}, Int. J. Math. Sciences, 2020 (2020).

\bibitem{avila2020notions} J. Ávila; V. Marín, \textit{The Notions of Center, Commutator and Inner Isomorphism for Groupoids}, Ingeniería y Ciencia, 16 (31) (2020), 7-26.


\bibitem{bagio2019partial} D. Bagio; A. Paques; H. Pinedo, \textit{On partial skew groupoids rings}, arXiv preprint arXiv:1905.12608 (2019).

 \bibitem{brandt1927verallgemeinerung} H. Brandt, \textit {Über eine verallgemeinerung des gruppenbegriffes}, Mathematische Annalen 96 (1) (1927), 360-366.

\bibitem{ivan2002algebraic} G. Ivan, \textit{Algebraic constructions of Brandt groupoids}, Proc. Alg. Symposium, Babes-Bolyai University, Cluj (2002), 69-90.



\bibitem{paques2018galois} A. Paques; T. Tamusiunas, \textit{The Galois Correspondence Theorem for Groupoid Actions}, J. Algebra, 509 (2018), 105-123.




\end{thebibliography}
{
}

\end{document}